\documentclass[12pt, reqno]{amsart}
\makeatletter
\@namedef{subjclassname@1991}{$\mathrm{1991}$ Mathematics Subject Classification}
\@namedef{subjclassname@2000}{$\mathrm{2000}$ Mathematics Subject Classification}
\@namedef{subjclassname@2010}{$\mathrm{2010}$ Mathematics Subject Classification}
\@namedef{subjclassname@2020}{$\mathrm{2020}$ Mathematics Subject Classification}
\makeatother
\usepackage{amsmath,amsthm, amscd, amsfonts, amssymb, graphicx, color}
\usepackage[bookmarksnumbered, colorlinks, plainpages,linkcolor=blue,urlcolor=blue,citecolor=blue]{hyperref}
\newtheorem{thm}{Theorem}[section]
\newtheorem{cor}[thm]{Corollary}
\newtheorem{lem}[thm]{Lemma}

\numberwithin{equation}{section}

\begin{document}

\title{New additive results on *-DMP elements in Banach algebras}

\author{Fatemeh Zamiri}
\author{Ali Ghaffari}
\author{Marjan Sheibani$^*$}
\address{
Department of Mathematics\\ Statistics and Computer Science, Semnan University, Semnan, Iran}
\email{<aghaffari@semnan.ac.ir>}
\address{
Department of Mathematics\\ Statistics and Computer Science, Semnan University, Semnan, Iran}
\email{<f.zamiri.math@gmail.com>}
\address{Farzanegan Campus, Semnan University, Semnan, Iran}
\email{<m.sheibani@semnan.ac.ir>}

\subjclass[2020]{47A05, 46L05, 15A09.}\keywords{pseudo core inverse; *-DMP element; additive property; block complex matrix; Banach algebra.}

\begin{abstract} We present new additive results for the *-DMP elements in a Banach algebra with involution. The necessary and sufficient conditions
under which the sum of two perturbed *-DMP elements in a Banach *-algebra is a *-DMP element are obtained. As an application, the *-DMP block complex matrices are investigated. These also extend the main results of Gao, Chen and Ke (Filomat, {\bf 32}(2018), 3073--3085).
\end{abstract}

\thanks{Corresponding author: Marjan Sheibani}

\maketitle

\section{Introduction}

An involution of a Banach algebra $\mathcal{A}$ is an anti-automorphism whose square is the identity map $1$. A Banach algebra $\mathcal{A}$ with involution $*$ is called a Banach *-algebra, e.g., $C^*$-algebra. Let $\mathcal{A}$ be a Banach *-algebra.
An element $a\in \mathcal{A}$ has Drazin inverse provided that there exists $x\in \mathcal{A}$ such that $$xa^{k+1}=a^k, ax^2=x, ax=xa,$$ where $k$ is the index of $a$ (denoted by $i(a)$), i.e., the smallest $k$ such that the previous equations are satisfied. Such $x$ is unique if exists, denoted by $a^{D}$, and called the Drazin inverse of $a$. An element $a\in \mathcal{A}$ has core inverse if and only if there exists $x\in \mathcal{A}$ such that $$a=axa, x\mathcal{A}=a\mathcal{A}, \mathcal{A}x=\mathcal{A}a^*.$$ If such $x$ exists, it is unique, and denote it by $a^{\tiny\textcircled{\#}}$ (see~\cite{Mi}).
 Recently, the notion of pseudo core inverse was introduced which is an extension of Drazin and core inverses (see~\cite{GC}). An element $a\in \mathcal{A}$ has p-core inverse (i.e., pseudo core inverse) if there exist $x\in \mathcal{A}$ and $k\in \Bbb{N}$ such that $$xa^{k+1}=a^k, ax^2=x, (ax)^*=ax.$$ If such $x$ exists, it is unique, and denote it by $a^{\tiny\textcircled{D}}$.

 An element $a$ in a *-Banach algebra $\mathcal{A}$ is *-DMP if there exists some $n\in {\Bbb N}$ such that $a^n$ has More-Penrose and group inverses and $(a^n)^{\dag}=(a^n)^{\#}$ (see~\cite{GC3}). We list several characterizations of *-DMP elements in a Banach *-algebra $\mathcal{A}$.

\begin{thm} (see~\cite[Theorem 2.2]{CZ} and \cite[Theorem 2.9]{GC3}) Let $\mathcal{A}$ be a Banach *-algebra, and let $a\in \mathcal{A}$. Then the following are equivalent:\end{thm}
\begin{enumerate}
\item [(1)]{\it $a$ is *-DMP.}
\item [(2)]{\it $a\in \mathcal{A}^D$ and $a^{\pi}$ is a projection.}
\item [(3)]{\it $a\in \mathcal{A}^{\tiny\textcircled{D}}$ and $a^{\tiny\textcircled{D}}=a^D$.}
\item [(4)]{\it $a, a^*\in \mathcal{A}^{\tiny\textcircled{D}}$ and $a^{\tiny\textcircled{D}}=(a^*)^{\tiny\textcircled{D}}$.}
\item [(5)]{\it There exists a projection $e\in \mathcal{A}$ such that $ae=ea, a+e\in \mathcal{A}^{-1}$ and $ae$ is nilpotent.}
\end{enumerate}

*-DMP elements in a Banach *-algebra were extensively studied by many authors (see~\cite{CZ,GC1,GC3, MD1, W,Z}). The
aim of this paper is to present new additive results for the *-DMP elements in a Banach *-algebra. The necessary and sufficient conditions
under which the sum of two perturbed *-DMP elements in a Banach *-algebra is a *-DMP element are established.
Let $a,b\in \mathcal{A}$ be *-DMP. If $ab=ba=0$ and $a^*b=0$, then $a+b\in \mathcal{A}$ is *-DMP (see~\cite[Theorem 2.16]{GC3}). In Section 2, we extend the preceding result and provide the necessary and sufficient conditions so that $a+b$ is *-DMP under perturbed conditions $a^{\pi}ab=a^{\pi}ba=0$ and $a^{\pi}a^*b=0$. In Section 3, we establish the equivalent conditions under which the sum of *-DMP elements is a *-DMP element under perturbed conditions $a^{\pi}ab=a^{\pi}ba$ and $a^{\pi}a^*b=a^{\pi}ba^*$. Let $C^{n\times n}$ be the Banach *-algebra of $n\times n$ complex matrices, with conjugate transpose as the involution.
As an application, in Section 4, we apply our preceding results to complex matrices and obtain various conditions under which a block complex matrix is *-DMP.

Throughout the paper, all Banach *-algebras are complex with an identity. $p\in \mathcal{A}$ is a projection provided that $p^2=p=p^*$.
We use $\mathcal{A}^{D}$ and $\mathcal{A}^{\tiny\textcircled{D}}$ to denote the set of all Drazin and p-core invertible elements in $\mathcal{A}$.
$A^*$ stands for the conjugate transpose $\overline{A}^T$ of the complex matrix $A$.

\section{Orthogonal perturbations}

This section is devoted to investigate the additive results of two *-DMP elements in a Banach *-algebra under the orthogonal perturbations. We begin with

\begin{lem} (see~\cite[Theorem 2.16]{GC3})) Let $a,b\in \mathcal{A}$ be *-DMP. If $ab=ba=0$ and $a^*b=0$, then $a+b\in \mathcal{A}$ is *-DMP.\end{lem}

Let $a, p^2=p\in \mathcal{A}$. Then $a$ has the Pierce decomposition relative to $p$, and we denote it by $\left(\begin{array}{cc}
a_{11}&a_{12}\\
a_{21}&a_{22}
\end{array}
\right)_p$. We now derive

\begin{lem} \end{lem}
\begin{enumerate}
\item [(1)] Let $x=\left(
\begin{array}{cc}
a&b\\
0&d
\end{array}
\right).$ Then $x\in M_2(\mathcal{A})$ is *-DMP if and only if $a,d\in \mathcal{A}$ are *-DMP and $$\sum\limits_{i=1}^{m}a^{m-i}bd^{i}d^{\pi}=0, \sum\limits_{i=1}^{m}a^ia^{\pi}bd^{m-i}=0$$ for some $m\geq max\{ i(a), i(d)\}$.
\item [(2)] Let $p$ be a projection and $x=\left(
\begin{array}{cc}
a&b\\
0&d
\end{array}
\right)_p.$ Then $x\in \mathcal{A}$ is *-DMP if and only if $a\in p\mathcal{A}p$ and $d\in p^{\pi}\mathcal{A}p^{\pi}$ are *-DMP and
$$\sum\limits_{i=1}^{m}a^{m-i}bd^{i}d^{\pi}=0, \sum\limits_{i=1}^{m}a^ia^{\pi}bd^{m-i}=0$$ for some $m\geq max\{ i(a), i(d)\}$.\end{enumerate}
\begin{proof} $(1)$ $\Longrightarrow $ In view of Theorem 1.1, $x\in \mathcal{A}^{\tiny\textcircled{D}}$. In view of ~\cite[Theorem 2.5]{GC}, $x^m\in \mathcal{A}^{\tiny\textcircled{\#}}$, where $m=i(x)$. Then $m\geq i(a)$. In this case, $(x^m)^{\tiny\textcircled{\#}}=(x^{\tiny\textcircled{D}})^m$ and $x^{\tiny\textcircled{D}}=x^{m-1}(x^m)^{\tiny\textcircled{\#}}$.
By using Theorem 1.1 again, $x^{\tiny\textcircled{D}}=x^D$. In light of ~\cite[Theorem 5.1]{DS},
we can write $x^{\tiny\textcircled{D}}=\left(
\begin{array}{cc}
*&*\\
0&*
\end{array}
\right),$ and so $(x^{\tiny\textcircled{D}})^m=\left(
\begin{array}{cc}
*&*\\
0&*
\end{array}
\right).$ This implies that $(x^m)^{\tiny\textcircled{\#}}=\left(
\begin{array}{cc}
*&*\\
0&*
\end{array}
\right).$ Obviously, we have $x^m=\left(
\begin{array}{cc}
a^m&b_m\\
0&d^m
\end{array}
\right)_p,$ where $b_1=b, b_m=ab_{m-1}+bd^{m-1}.$
By induction,we verify that $$b_m=\sum\limits_{i=1}^{m}a^{i-1}bd^{m-i}=\sum\limits_{i=1}^{m}a^{m-i}bd^{i-1}.$$
In light of~\cite[Theorem 2.5]{XS}, $a^m,d^m\in \mathcal{A}^{\tiny\textcircled{\#}}$ and $(a^m)^{\pi}b_m=0$.
This implies that $$\sum\limits_{i=1}^{m}a^{i-1}a^{\pi}bd^{m-i}=a^{\pi}\sum\limits_{i=1}^{m}a^{i-1}bd^{m-i}=0.$$

Since $x^m\in \mathcal{A}_{\tiny\textcircled{\#}}$, it follows by~\cite[Theorem 2.9]{XS} that $b_m(d^m)^{\pi}=0$, i.e., $b_md^{\pi}=0$.
Then $\sum\limits_{i=1}^{m}a^{m-i}bd^{i-1}d^{\pi}=0$.
In view of ~\cite[Theorem 2.5]{GC}, $a,d\in \mathcal{A}^{\tiny\textcircled{D}}$. Moreover, we have $a^{\tiny\textcircled{D}}=a^D$. It follows by Theorem 1.1 that
$a$ is *-DMP. Likewise, $d\in \mathcal{A}$ is *-DMP, as desired.

$\Longleftarrow $ Since $a,d\in \mathcal{A}$ are *-DMP, it follows by Theorem 1.1 and ~\cite[Theorem 2.5]{GC} that $a^k,d^k\in \mathcal{A}^{\tiny\textcircled{\#}}\bigcap \mathcal{A}_{\tiny\textcircled{\#}}$, where $k=max\{i(a),i(b)\}$. Write $x^k=\left(
\begin{array}{cc}
a^k&b_k\\
0&d^k
\end{array}
\right),$ where $b_1=b, b_k=ab_{k-1}+bd^{k-1}.$ As in the preceding discussion,
$a^{\pi}b_k=0$ and $b_kd^{\pi}=0$ and $b_k=\sum\limits_{i=1}^{k}a^{i-1}bd^{k-i}=\sum\limits_{i=1}^{k}a^{k-i}bd^{i-1}$. In light of~\cite[Theorem 2.5]{XS}, $x^k\in \mathcal{A}^{\tiny\textcircled{\#}}$. In this case,
$$(x^k)^{\tiny\textcircled{\#}}=\left(
\begin{array}{cc}
(a^k)^{\tiny\textcircled{\#}}&-(a^k)^{\tiny\textcircled{\#}}b_k(d^k)^{\tiny\textcircled{\#}}\\
0&(d^k)^{\tiny\textcircled{\#}}
\end{array}
\right).$$ Thus, we have
$$\begin{array}{rll}
x^{\tiny\textcircled{D}}&=&\left(
\begin{array}{cc}
a^{\tiny\textcircled{D}}&-a^{\tiny\textcircled{D}}b_kd^{\tiny\textcircled{D}}\\
0&d^{\tiny\textcircled{D}}
\end{array}
\right)\\
&=&\left(
\begin{array}{cc}
a^D&-a^Db_kd^D\\
0&d^D
\end{array}
\right).
\end{array}$$
In view of~\cite[Theorem 5.1]{DS}, we have
$x^D=\left(
\begin{array}{cc}
a^D&z\\
0&d^D
\end{array}
\right),$ where $$\begin{array}{rll}
z&=&\sum\limits_{n=0}^{r(d)-1}(a^D)^{n+2}b_kd^{\pi}d^n+\sum\limits_{n=0}^{r(a)-1}a^na^{\pi}b_k(d^D)^{n+2}-a^Db_kd^D\\
&=&-a^Db_kd^D.
\end{array}$$ We infer that $x^{\tiny\textcircled{D}}=x^D$. According to Theorem 1.1, $x\in \mathcal{A}$ is *-DMP.

$(2)$ $\Longrightarrow $ Since $x\in \mathcal{A}^{\tiny\textcircled{D}}$, we have $x^m\in \mathcal{A}^{\tiny\textcircled{\#}}$, where $m=i(x)\geq i(a)$. Moreover,
$(x^m)^{\tiny\textcircled{\#}}=(x^{\tiny\textcircled{D}})^m$ and $x^{\tiny\textcircled{D}}=x^{m-1}(x^m)^{\tiny\textcircled{\#}}$.
Since $p^{\pi}x^{\tiny\textcircled{D}}p=0$, we can write $x^{\tiny\textcircled{D}}=\left(
\begin{array}{cc}
*&*\\
0&*
\end{array}
\right)_p,$ and so $(x^{\tiny\textcircled{D}})^m=\left(
\begin{array}{cc}
*&*\\
0&*
\end{array}
\right)_p.$ This implies that $p^{\pi}(x^{\tiny\textcircled{D}})^mp=0.$ Hence $p^{\pi}(x^m)^{\tiny\textcircled{\#}}p=0.$
Write $x^m=\left(
\begin{array}{cc}
a^m&b_m\\
0&d^m
\end{array}
\right)_p,$ where $b_1=b, b_m=ab_{m-1}+bd^{m-1}.$
Then $b_m=\sum\limits_{i=1}^{m}a^{i-1}bd^{m-i}$.
Similarly to ~\cite[Theorem 2.5]{XS}, we see that $a^m,d^m\in p\mathcal{A}^{\tiny\textcircled{\#}}p\subseteq (p\mathcal{A}p)^{\tiny\textcircled{\#}}$ and $(a^m)^{\pi}b_m=0$. In view of ~\cite[Theorem 2.5]{GC}, $a\in (p\mathcal{A}p)^{\tiny\textcircled{D}}$ and $d\in (p^{\pi}\mathcal{A}p^{\pi})^{\tiny\textcircled{D}}$.
As in the preceding discussion, we have $$\sum\limits_{i=1}^{m}a^{m-i}bd^{i}d^{\pi}=0, \sum\limits_{i=1}^{m}a^ia^{\pi}bd^{m-i}=0$$ for some $m\geq max\{ i(a), i(d)\}$.

$\Longleftarrow $ Since $a\in (p\mathcal{A}p)^{\tiny\textcircled{D}},d\in (p^{\pi}\mathcal{A}p^{\pi})^{\tiny\textcircled{D}}$, it follows by~\cite[Theorem 2.5]{GC} that $a^k\in (p\mathcal{A}p)^{\tiny\textcircled{\#}},d^k\in (p^{\pi}\mathcal{A}p^{\pi})^{\tiny\textcircled{\#}}$, where $k=max\{i(a),i(b)\}$. Write $x^k=\left(
\begin{array}{cc}
a^k&b_k\\
0&d^k
\end{array}
\right)_p,$ where $b_1=b, b_k=ab_{k-1}+bd^{k-1}.$ Then $b_m=\sum\limits_{i=1}^{m}a^{i-1}bd^{m-i}$.
By hypothesis, we have $a^{\pi}b_m=0$. As in the preceding discussion,
we have $a^{\pi}b_k=0$. In light of~\cite[Theorem 2.5]{XS}, $x^k\in \mathcal{A}^{\tiny\textcircled{\#}}$. According to~\cite[Theorem 2.5]{GC}, $x\in \mathcal{A}^{\tiny\textcircled{D}}$. Furthermore, $p^{\pi}(x^k)^{\tiny\textcircled{\#}}p=0$, and so $p^{\pi}(x^{\tiny\textcircled{D}})^kp=0$,
Write $(x^{\tiny\textcircled{D}})^k=\left(
\begin{array}{cc}
*&*\\
0&*
\end{array}
\right)_p$. Then $x^{\tiny\textcircled{D}}=x^{k-1}(x^{\tiny\textcircled{D}})^k=\left(
\begin{array}{cc}
*&*\\
0&*
\end{array}
\right)_p.$ This implies that $p^{\pi}x^{\tiny\textcircled{D}}p=0$, as asserted.\end{proof}

We come now to the demonstration for which this section has been developed.

\begin{thm} Let $a,b, a^{\pi}b\in \mathcal{A}$ be *-DMP. If $a^{\pi}ab=a^{\pi}ba=a^{\pi}a^*b=0$, then the following are equivalent:\end{thm}
\begin{enumerate}
\item [(1)] $a+b\in \mathcal{A}$ is *-DMP.
\vspace{-.5mm}
\item [(2)] $(a+b)aa^D\in \mathcal{A}$ is *-DMP and $$\begin{array}{rll}
\sum\limits_{i=1}^{m}(a+b)^{m-i}[aa^D,b](a^i+b^i)a^{\pi}b^{\pi}&=&0,\\
\sum\limits_{i=1}^{m}(a+b)^i(a+b)^{\pi}[aa^D,b](a^{m-i}+b^{m-i})&=&0,
\end{array}$$ for some $m\geq max\{i(a), i(a+b)\}$.
\end{enumerate}
\begin{proof} Let $p=aa^D$. By hypothesis, $p^{\pi}bp=(1-aa^D)baa^D=(a^{\pi}ba)a^D=0$. So we get $$a=\left(
\begin{array}{cc}
a_1&0\\
0&a_4
\end{array}
\right)_p, b=\left(
\begin{array}{cc}
b_1&b_2\\
0&b_4
\end{array}
\right)_p.$$
Hence $$a+b=\left(
\begin{array}{cc}
a_1+b_1&b_2\\
0&a_4+b_4
\end{array}
\right)_p.$$

We see that $a_1=a^2a^D$ and $b_1=aa^Dbaa^D=baa^D$. Then
$$a_1+b_1=(a+b)aa^D.$$ Moreover, we check that
$$\begin{array}{rll}
(a_1+b_1)^{i}&=&(a+b)^iaa^D,\\
(a_1+b_1)^D&=&(a+b)^Daa^D,\\
(a_1+b_1)^{\pi}&=&1-(a+b)(a+b)^Daa^D.
\end{array}$$

Also we have $a_4=aa^{\pi}$ and $b_4=a^{\pi}ba^{\pi}=a^{\pi}b$, and so
$$a_4+b_4=a^{\pi}a+a^{\pi}b.$$

Then we check that
$$\begin{array}{rll}
(a_4+b_4)^{i}&=&a^{\pi}(a^i+b^i),\\
(a_4+b_4)^D&=&a^{\pi}b^D,\\
(a_4+b_4)^{\pi}&=&1-a^{\pi}bb^D.
\end{array}$$

Clearly, $a^{\pi}a$ is *-DMP. Moreover, we have
$$\begin{array}{rll}
(a^{\pi}a)(a^{\pi}b)&=&0, \\
(a^{\pi}b)(a^{\pi}a)&=&0,\\
(a^{\pi}a)^*(a^{\pi}b)&=&0.
\end{array}$$ In view of Lemma 2.1, $a_4+b_4$ is *-DMP.

In light of Lemma 2.2, $a+b$ is *-DMP if and only if
$a_1+b_1$ is *-DMP and
$$\begin{array}{rll}
\sum\limits_{i=1}^{m}(a_1+b_1)^{m-i}b_2(a_4+b_4)^{i}(a_4+b_4)^{\pi}&=&0,\\
\sum\limits_{i=1}^{m}(a_1+b_1)^i(a_1+b_1)^{\pi}b_2(a_4+b_4)^{m-i}&=&0,
\end{array}$$ for some $m\geq max\{ i(a_1+b_1), i(a_4+b_4)\}$.
Therefore $a+b$ is *-DMP if and only if $(a+b)aa^D$ is *-DMP and
$$\begin{array}{rll}
\sum\limits_{i=1}^{m}(a+b)^{m-i}[aa^D,b](a^i+b^i)a^{\pi}b^{\pi}&=&0,\\
\sum\limits_{i=1}^{m}(a+b)^i(a+b)^{\pi}[aa^D,b](a^{m-i}+b^{m-i})&=&0,
\end{array}$$ where $m\geq max\{i(a), i(a+b)\}$, This completes the proof.\end{proof}

An element $a$ in $\mathcal{A}$ is EP if $a\in \mathcal{A}^{\#}\bigcap \mathcal{A}^{\dag}$ and $a^{\#}=a^{\dag}$. Evidently, $a$ in $\mathcal{A}$ is EP if and only if
$a\in \mathcal{A}^{\#}$ and $(aa^{\#})^*=aa^{\#}$ (see~\cite[Theorem 1.2]{MDK}).

\begin{cor} Let $a,b,a^{\pi}b\in \mathcal{A}$ be EP. If $a^{\pi}ba=0$, then the following are equivalent:\end{cor}
\begin{enumerate}
\item [(1)] $a+b\in \mathcal{A}$ is EP.
\vspace{-.5mm}
\item [(2)] $(a+b)aa^{\#}\in \mathcal{A}$ is EP and $$\begin{array}{rll}
\sum\limits_{i=1}^{m}(a+b)^{m-i}[aa^D,b](a^i+b^i)a^{\pi}b^{\pi}&=&0,\\
\sum\limits_{i=1}^{m}(a+b)^i(a+b)^{\pi}[aa^D,b](a^{m-i}+b^{m-i})&=&0,
\end{array}$$ where $m\geq max\{i(a), i(a+b)\}$.
\end{enumerate}
\begin{proof} $(1)\Rightarrow (2)$ Since $a+b\in \mathcal{A}$ is EP, it is *-DMP and has group inverse.
In view of Theorem 2.3, $(a+b)aa^{\#}\in \mathcal{A}$ is *-DMP and $$\begin{array}{rll}
\sum\limits_{i=1}^{m}(a+b)^{m-i}[aa^D,b](a^i+b^i)a^{\pi}b^{\pi}&=&0,\\
\sum\limits_{i=1}^{m}(a+b)^i(a+b)^{\pi}[aa^D,b](a^{m-i}+b^{m-i})&=&0,
\end{array}$$ where $m\geq max\{i(a), i(a+b)\}$. In view of ~\cite[Theorem 2.1]{MD}, $(a+b)aa^{\#}\in \mathcal{A}^{\#}$.
Set $w=(a+b)aa^{\#}$. By virtue of Theorem 1.1, $w^{\tiny\textcircled{D}}=w^D=w^{\#}.$ Moreover,
$w^{\tiny\textcircled{D}}w=w^Dw$ is a projection, and then so is $w^{\#}w$. According to Theorem 1.1, $w$ is EP, as desired.

$(2)\Rightarrow (1)$ Since $w=(a+b)aa^{\#}\in \mathcal{A}$ is EP, it is *-DMP. In light of Theorem 2.3, $a+b\in \mathcal{A}$ is *-DMP.
In view of ~\cite[Theorem 2.1]{MD}, $a+b$ has group inverse. As in the preceding discussion, $a+b$ is EP, and therefore we complete the proof.\end{proof}

\section{commutating perturbations}

In this section we determine necessary and sufficient conditions under which the sum of two commutating perturbed *-DMP elements is *-DMP. For future use,, we now record the following.

\begin{lem} (see~\cite[Theorem 2.15]{GC3}) Let $a,b\in \mathcal{A}$ be *-DMP. If $ab=ba$ and $a^*b=ba^*$, then $ab\in \mathcal{A}$ is *-DMP.\end{lem}

\begin{thm} Let $a,b\in \mathcal{A}$ be *-DMP. If $ab=ba$ and $a^*b=ba^*$, then the following are equivalent:\end{thm}
\begin{enumerate}
\item [(1)] $a+b\in \mathcal{A}$ is *-DMP.
\vspace{-.5mm}
\item [(2)] $1+a^Db\in \mathcal{A}$ is *-DMP.
\end{enumerate}
\begin{proof} $(1)\Rightarrow (2)$ In view of Theorem 1.1, $a+b\in \mathcal{A}^D$. By virtue of ~\cite[Theorem 3]{ZC},
  $1+a^Db\in \mathcal{A}^D$ and $$(1+a^Db)^D=a^{\pi}+a^2a^D(a+b)^D.$$
Then $$\begin{array}{rll}
(1+a^Db)(1+a^Db)^D&=&a^{\pi}+(1+a^Db)a^2a^D(a+b)^D\\
&=&a^{\pi}+aa^D(a+b)(a+b)^D\\
&=&1-aa^D(a+b)^{\pi}.
\end{array}$$ Therefore $1+a^Db\in \mathcal{A}$ is *-DMP by Theorem 1.1.

$(2)\Rightarrow (1)$ Since $1+a^Db\in \mathcal{A}$ is *-DMP, it follows by Theorem 1.1 that $1+a^Db\in \mathcal{A}^D$. In view of ~\cite[Theorem 3]{ZC},
$a+b\in \mathcal{A}^D$ and $$(a+b)^D=(1+a^Db)^Da^D+b^D(1+aa^{\pi}b^D)^{-1}a^{\pi}.$$ Since $(1-a^{\pi}bb^D)(1+aa^{\pi}b^D)=1-a^{\pi}bb^D$, we have
$$(1-a^{\pi}bb^D)(1+aa^{\pi}b^D)^{-1}=1-a^{\pi}bb^D.$$ Then we check that
$$\begin{array}{rl}
&(a+b)(1+b)^D\\
=&a^D(a+b)(1+a^Db)^D+(a+b)a^{\pi}b^D(1+aa^{\pi}b^D)^{-1}\\
=&aa^D(1+a^Db)(1+a^Db)^D+(aa^{\pi}b^D+a^{\pi}bb^D)(1+aa^{\pi}b^D)^{-1}\\
=&aa^D(1+a^Db)(1+a^Db)^D+1-(1-a^{\pi}bb^D)(1+aa^{\pi}b^D)^{-1}\\
=&aa^D(1+a^Db)(1+a^Db)^D+1-[1-a^{\pi}bb^D]\\
=&aa^D(1+a^Db)(1+a^Db)^D+a^{\pi}bb^D.\\
\end{array}$$
Therefore $$\begin{array}{rll}
(a+b)^{\pi}&=&1-aa^D(1+a^Db)(1+a^Db)^D-a^{\pi}bb^D\\
&=&aa^D-aa^D(1+a^Db)(1+a^Db)^D+a^{\pi}b^{\pi}\\
&=&aa^D(1+a^Db)^{\pi}+a^{\pi}b^{\pi}.
\end{array}$$ Hence, $(a+b)^{\pi}$ is a projection.
Accordingly, $a+b\in \mathcal{A}$ is *-DMP by Theorem 1.1.\end{proof}

We are now ready to prove:

\begin{thm} Let $a,b,a^{\pi}b\in \mathcal{A}$ be *-DMP. If $a^{\pi}ab=a^{\pi}ba$ and $a^{\pi}a^*b=a^{\pi}ba^*$, then the following are equivalent:\end{thm}
\begin{enumerate}
\item [(1)] $a+b\in \mathcal{A}$ is *-DMP.
\vspace{-.5mm}
\item [(2)] $(a+b)aa^D\in \mathcal{A}$ is *-DMP and $$\begin{array}{rll}
\sum\limits_{i=1}^{m}(a+b)^{m-i}[aa^D,b](a^i+b^i)a^{\pi}b^{\pi}&=&0,\\
\sum\limits_{i=1}^{m}(a+b)^i(a+b)^{\pi}[aa^D,b](a^{m-i}+b^{m-i})&=&0,
\end{array}$$ where $m\geq max\{i(a), i(a+b)\}$.
\end{enumerate}
\begin{proof} Since $a^{\pi}ab=a^{\pi}ba$, we have $a(a^{\pi}b)=(a^{\pi}b)a$. In view of ~\cite[Theorem 2.2]{D},
$a^D(a^{\pi}b)=(a^{\pi}b)a^D$. Hence, $a^{\pi}ba^D=0$. Let $p=aa^D$. Then $p^{\pi}bp=(a^{\pi}ba^D)a=0$, and then we have $$a=\left(
\begin{array}{cc}
a_1&0\\
0&a_4
\end{array}
\right)_p, b=\left(
\begin{array}{cc}
b_1&b_2\\
0&b_4
\end{array}
\right)_p.$$
Thus, $$a+b=\left(
\begin{array}{cc}
a_1+b_1&b_2\\
0&a_4+b_4
\end{array}
\right)_p.$$
Here, $a_1=a^2a^D$ and $b_1=aa^Dbaa^D=(1-a^{\pi})baa^D=baa^D$. Then
$$a_1+b_1=(a+b)aa^D.$$ We verify that
$$\begin{array}{rll}
(a_1+b_1)^{i}&=&(a+b)^iaa^D,\\
(a_1+b_1)^D&=&(a+b)^Daa^D,\\
(a_1+b_1)^{\pi}&=&1-(a+b)(a+b)^Daa^D.
\end{array}$$

Further, we have $a_4=aa^{\pi}$ and $b_4=a^{\pi}ba^{\pi}=a^{\pi}b$; hence,
$$a_4+b_4=a^{\pi}a+a^{\pi}b.$$

Then we check that
$$\begin{array}{rll}
(a_4+b_4)^{i}&=&a^{\pi}(a^i+b^i),\\
(a_4+b_4)^{\pi}&=&(aa^{\pi})^{\pi}(a^{\pi}b)^{\pi}=1-a^{\pi}bb^D.
\end{array}$$

Obviously, $a^{\pi}a$ is *-DMP. Moreover, we have
$$\begin{array}{rll}
(a^{\pi}a)(a^{\pi}b)&=&a^{\pi}ab=a^{\pi}ba=(a^{\pi}b)(a^{\pi}a), \\
(a^{\pi}a)^*(a^{\pi}b)&=&a^{\pi}a^*b=a^{\pi}ba^*=a^{\pi}ba^{\pi}a^*=(a^{\pi}b)(a^{\pi}a)^*,\\
1+(a^{\pi}a)^D(a^{\pi}b)&=&1~\mbox{is *-DMP}.
\end{array}$$ By virtue Theorem 3.2, $a_4+b_4$ is *-DMP.

In light of Theorem 2.2, $a+b$ is *-DMP if and only if
$a_1+b_1$ is *-DMP and
$$\begin{array}{rll}
\sum\limits_{i=1}^{m}(a_1+b_1)^{m-i}b_2(a_4+b_4)^{i}(a_4+b_4)^{\pi}&=&0,\\
\sum\limits_{i=1}^{m}(a_1+b_1)^i(a_1+b_1)^{\pi}b_2(a_4+b_4)^{m-i}&=&0,
\end{array}$$ for some $m\geq max\{ i(a_1+b_1), i(a_4+b_4)\}$.

For any $i\in {\Bbb N}$, we easily check that $$\begin{array}{rl}
&(a_1+b_1)^{m-i}b_2(a_4+b_4)^{i}(a_4+b_4)^{\pi}\\
=&(a+b)^{m-i}aa^Db(a^i+b^i)a^{\pi}b^{\pi},\\
&(a_1+b_1)^i(a_1+b_1)^{\pi}b_2(a_4+b_4)^{m-i}\\
=&(a+b)^i(a+b)^{\pi}aa^Dba^{\pi}(a^{m-i}+b^{m-i}).
\end{array}$$ This completes the proof.
\end{proof}

As an immediate consequence of Theorem 3.3, we now derive

\begin{cor} Let $a,b\in \mathcal{A}$ be *-DMP. If $ab=ba$ and $a^{\pi}a^*b=a^{\pi}ba^*$, then the following are equivalent:\end{cor}
\begin{enumerate}
\item [(1)] $a+b\in \mathcal{A}$ is *-DMP.
\vspace{-.5mm}
\item [(2)] $(a+b)aa^D\in \mathcal{A}$ is *-DMP.
\end{enumerate}
\begin{proof} $(1)\Rightarrow (2)$ In view of Lemma 3.1, $a^{\pi}b$ is *-DMP. Thus $(a+b)aa^D\in \mathcal{A}$ is *-DMP by Theorem 3.3.

$(2)\Rightarrow (1)$ Clearly, $a^{\pi}b$ is *-DMP. Since $ab=ba$, it follows by ~\cite[Theorem 2.2]{D} that
$a^Db=ba^D$. Then $[aa^D,b]=0$, thus yielding the result by Theorem 3.3\end{proof}

\section{applications}

Let $M=\left(
  \begin{array}{cc}
    A&B\\
    C&D
  \end{array}
\right)\in {\Bbb C}^{2n\times 2n},$ where $A,D,BC,CB\in {\Bbb C}^{n\times n}$ are *-DMP. We shall apply the preceding results
to the block matrix $M$ and establish various conditions under which $M$ is *-DMP.

\begin{lem} $\left(
  \begin{array}{cc}
    0 & B \\
    C & 0
  \end{array}
\right)$ is *-DMP.\end{lem}
\begin{proof} Let $Q=\left(
  \begin{array}{cc}
    0 & B \\
    C & 0
  \end{array}
\right)$. Then $Q^2=\left(
  \begin{array}{cc}
  BC&0\\
  0&CB
  \end{array}
\right).$ By hypothesis, $Q^2$ is *-DMP. In view of~\cite[Theorem 2.14]{GC3},
$Q$ is *-DMP, as asserted.
\end{proof}

\begin{thm} If $AB=BD, DC=CA, A^*B=BD^*, D^*C=CA^*$ and $A^DBD^DC$ is nilpotent, then $M$ is *-DMP.\end{thm}
\begin{proof} Write $M=P+Q$, where $$P=\left(
  \begin{array}{cc}
    A & 0 \\
    0 & D
  \end{array}
\right), Q=\left(
  \begin{array}{cc}
   0 & B \\
   C & 0
  \end{array}
\right).$$ Since $A$ and $D$ are *-DMP, so is $P$. In view of Lemma 3.1, $Q$ is *-DMP. We easily check that
$$PQ=\left(
  \begin{array}{cc}
   0 & 0 \\
    DC & 0
  \end{array}
\right)=\left(
  \begin{array}{cc}
  0 & 0 \\
    CA & 0
  \end{array}
\right)=QP.$$ Likewise, we verify that $P^*Q=QP^*$. Moreover, we check that
$$\begin{array}{rll}
I_{m+n}+P^DQ&=&\left(
\begin{array}{cc}
I_m&A^DB\\
D^DC&I_n
\end{array}
\right).
\end{array}$$ Since $A^DBD^DC$ is nilpotent, we prove that
$I_{m+n}+P^DQ$ is invertible, and so it is *-DMP. According to Theorem 3.2, $M$ is *-DMP, as asserted.
\end{proof}

As an immediate consequence we now derive

\begin{cor} If $AB=BD, DC=CA, A^*B=BD^*, D^*C=CA^*,$ and $A^DBD^DC$ is nilpotent, then $M$ is *-DMP.\end{cor}
\begin{proof} Applying Theorem 4.2, $\left(
\begin{array}{cc}
D&C\\
B&A
\end{array}
\right)$ is *-DMP. Clearly,
$$M=\left(
\begin{array}{cc}
0&I\\
I&0
\end{array}
\right)\left(
\begin{array}{cc}
D&C\\
B&A
\end{array}
\right)\left(
\begin{array}{cc}
0&I\\
I&0
\end{array}
\right),$$ and so $M$ is *-DMP.\end{proof}

We are now ready to prove the following.

\begin{thm} If $AB=BD, DC=CA, B^*A=DB^*$ and $A^DBD^DC$ is nilpotent, then $M$ is *-DMP.\end{thm}
\begin{proof} Write $M=P+Q$, where $$P=\left(
  \begin{array}{cc}
    A & 0 \\
    0 & D
  \end{array}
\right), Q=\left(
  \begin{array}{cc}
   0 & B \\
   C & 0
  \end{array}
\right).$$ Then we check that
$$\begin{array}{rll}
Q^*P&=&\left(
  \begin{array}{cc}
   0 & C^* \\
  B^* & 0
  \end{array}
\right)\left(
  \begin{array}{cc}
    A & 0 \\
    0 & D
  \end{array}
\right)\\
&=&\left(
  \begin{array}{cc}
    0 & C^*D \\
    B^*A & 0
  \end{array}
\right)\\
&=&\left(
  \begin{array}{cc}
    0 & AC^* \\
   DB^* & 0
  \end{array}
\right)\\
&=&\left(
  \begin{array}{cc}
    A & 0 \\
    0 & D
  \end{array}
\right)\left(
  \begin{array}{cc}
   0 & C^* \\
  B^* & 0
  \end{array}
\right)\\
&=&PQ^*.
\end{array}$$ Similarly, $QP=PQ$.
Moreover, we verify that
$$\begin{array}{rll}
I+Q^{\tiny\textcircled{D}}P&=&I+\left(
  \begin{array}{cc}
  0& (BC)^{D}BCC^{\tiny\textcircled{D}} \\
  (CB)^{D}CBB^{\tiny\textcircled{D}}&0
  \end{array}
\right)\left(
  \begin{array}{cc}
    A & 0 \\
    0 & D
  \end{array}
\right)\\
&=&\left(
\begin{array}{cc}
I&A^{\tiny\textcircled{D}}(BC)^{D}BCC^{\tiny\textcircled{D}}D\\
(CB)^{D}CBB^{\tiny\textcircled{D}}A&I
\end{array}
\right).
\end{array}$$ Since $A^DBD^DC$ is nilpotent, we prove that
$I+Q^DP$ is invertible; hence, it is *-DMP. In light of Theorem 3.2, $M$ is *-DMP.\end{proof}

\begin{cor} If $AB=BD, DC=CA, A^*B=BD^*, D^*C=CA^*$ and $A^DBD^DC$ is nilpotent, then $M$ is *-DMP.\end{cor}
\begin{proof} As in the proof of Corollary 4.3, we obtain the result by applying Theorem 4.4 to the block matrix
$\left(
\begin{array}{cc}
D&C\\
B&A
\end{array}
\right)$.\end{proof}

\begin{thm} If $BC=0, CB=0, CA=DC, AC^*=C^*D$ and $$\sum\limits_{i=1}^{i(D)}A^{m-i}BD^{i-1}D^{\pi}=0, \sum\limits_{i=1}^{i(A)}A^{i-1}A^{\pi}BD^{m-i}=0,$$ then $M$ is *-DMP.\end{thm}
\begin{proof} Write $M=P+Q$, where $$P=\left(
  \begin{array}{cc}
   0 & 0 \\
    C & 0
  \end{array}
\right), Q=\left(
  \begin{array}{cc}
   A & B \\
   0 & D
  \end{array}
\right).$$ Clearly, $P$ is *-DMP. By using Theorem 2.2, $Q$ is *-DMP.
One easily checks that
$$\begin{array}{rll}
PQ&=&\left(
  \begin{array}{cc}
   0 & 0 \\
    CA & CB
  \end{array}
\right)\\
&=&\left(
  \begin{array}{cc}
    BC & 0 \\
    DC & 0
  \end{array}
\right)\\
&=&QP;\\
P^*Q&=&\left(
  \begin{array}{cc}
    0 & C^*D \\
    0 & 0
  \end{array}
\right)\\
&=&\left(
  \begin{array}{cc}
    0 & AC^* \\
   0 & 0
  \end{array}
\right)\\
&=&QP^*.
\end{array}$$ Clearly $P^D=0$, and so $I+P^DA=I$ is *-DMP.
Accordingly, $M$ is *-DMP by Theorem 3.2.\end{proof}

\begin{cor} If $BC=0, CB=0, AB=BD, B^*A=DB^*$ and $$\sum\limits_{i=1}^{i(A)}D^{m-i}CA^{i-1}A^{\pi}=0, \sum\limits_{i=1}^{i(D)}D^{i-1}D^{\pi}CA^{m-i}=0,$$ then $M$ is *-DMP.\end{cor}
\begin{proof} Applying Theorem 4.6 to the block matrix
$\left(
\begin{array}{cc}
D&C\\
B&A
\end{array}
\right)$, analogously to Corollary 4.3, we complete the proof.\end{proof}


\vskip10mm\begin{thebibliography}{99}
\bibitem{CM} H. Chen and M. Sheibani Abdolyousefi, World Scientific, Hackensack, NJ,
2022. https://doi.org/10.1142/12959.


 \bibitem{CZ} H. Chen and H. Zou, On *-DMP inverses in a ring with involution, {\it Commun. Algebra}, {\bf 49}(2021), 5006--5016.

\bibitem{CCZ} X. Chen; J. Chen and Y. Zhou, The pseudo core inverses of differences and products of projections in rings with involution,
{\it Filomat}, {\bf 35}(2021), 181--189.

\bibitem{DS} D.S. Djordjevic and P.S. Stanimirovic, On the generalized Drazin inverse and generalized resolvent,
 {\it Czechoslovak Math. J.}, {\bf 51(126)} (2001), 617--634.

\bibitem{D} M.P. Drazin, Commuting properties of generalized inverses,
{\it Linear Multilinear Algebra}, {\bf 61}(2013), 1675--1681.

\bibitem{GC} Y. Gao and J. Chen, Pseudo core inverses in rings with involution, {\it Comm. Algebra}, {\bf 46}(2018), 38--50.

\bibitem{GC1} Y. Gao and J. Chen, Characterizations of *-DMP matrices over rings, {\it Turk. J. Math.}, {\bf 42}(2018), 786--796.

\bibitem{GC2} Y. Gao and J. Chen, The pseudo core inverse of a lower triangular matrix, {\it Rev. R. Acad. Cienc. Exactas Fis. Nat., Ser. A Mat.}, {\bf 113}(2019), 423--434.

\bibitem{GC3} Y. Gao; J. J. Chen and Y. Ke, *-DMP elements in *-semigroups and *-rings,  {\it Filomat}, {\bf 32}(2018), 3073--3085.

\bibitem{MD} N. Mihajlovic and D.S. Djordjevic, On group invrtibility in rings,
{\it Filomat}, {\bf 33}(2019), 6141--6150.

\bibitem{Mi} N. Mihajlovic, Group inverse and core inverse in Banach and $C^*$-algebras,
{\it Comm. Algebra}, {\bf 48}(2020), 1803--1818.

\bibitem{M} D. Mosic, Core-EP inverses in Banach algebras,
{\it Linear Multilinear Algebra}, {\bf 69}(2021), 2976--2989.

\bibitem{MD1} M. Mosic and D.S. Djordjevic, EP elements in Banach algebras, {\it Filomat} {\bf 5}(2011), 25--32.


\bibitem{MDK} D. Mosic; D.S, Djordjevic and J.J. Koliha, EP elements in rings,
 {\it Linear Algebra Appl.}, {\bf 4319}(2009), 527--535.

\bibitem{W} L. Wang; D. Mosic and Y. Gao, New results on EP elements in rings with involution, {\it Algebra Colloq.}, {\bf 29}(2022), 39--52.

 \bibitem{XS} S. Xu, Core invertibility of triangular matrices over a ring, {\it Indian J. Pure Appl. Math.},
 {\bf 50}(2019), 837-847.

\bibitem{Z2} M. Zhou; J. Chen and D. Wang, The core inverses of linear combinations of two core invertible matrices,
{\it Linear Multilinear Algebra}, {\bf 69}(2021), 702--718.

\bibitem{Z3} M. Zhou; J. Chen and Y. Zhou, Weak group inverses in proper *-rings,
{\it J. Algebra Appl.}, {\bf 19}(2020), https://doi.org/10.1142/S0219498820502382.

\bibitem{Z} H. Zhu, On DMP inverses and m-EP elements in rings,
 {\it Linear Multilinear Algebra}, {\bf 67}(2018), 756--766.

 \bibitem{ZC} G. Zhuang; J. Chen; D.S. Cvetkovic-Ilic and Y. Wei, Additive property of Drazin invertibility of elements in a ring, {\it Linear and Multilinear Algebra}, {\bf 60}(2012), 903--910.

\end{thebibliography}
\end{document}